\documentclass{amsart}
\usepackage{amsmath}
\usepackage{amsfonts}
\usepackage[colorlinks=true,linkcolor=red,citecolor=blue]{hyperref}
\usepackage{amssymb}
\usepackage{graphicx}
\providecommand{\U}[1]{\protect\rule{.1in}{.1in}}
\newtheorem{theorem}{Theorem}
\theoremstyle{plain}

\newtheorem{case}{Case}

\newtheorem{corollary}{Corollary}

\newtheorem{example}{Example}

\newtheorem{lemma}{Lemma}

\newtheorem{remark}{Remark}

\numberwithin{equation}{section}

\begin{document}
\title[Abelian theorem]{A sharp Abelian theorem for the Laplace transform}
\author{Maeva Biret$^{(1)}$}
\author{Michel Broniatowski$^{(1,\ast)}$}
\author{Zansheng Cao$^{(1)}$}
\curraddr{$^{(1)}$Universit\'{e} Pierre et Marie Curie, Paris}
\email{$^{(\ast)}$ Corresponding author: michel.broniatowski@upmc.fr}
\date{March 5th, 2014}

\begin{abstract}
This paper states asymptotic equivalents for the moments of the Esscher
transform of a distribution on $\mathbb{R}$ with smooth density in the upper
tail. As a by product if provides a tail approximation for its moment
generating function, and shows that the Esscher transforms have a Gaussian
behavior for large values of the parameter. 
\end{abstract}
\maketitle
\section{Introduction}
\label{SECTION INTRO} Let $X$ denote a real-valued random variable with
support $\mathbb{R}$ and distribution $P_{X}$ with density $p$.

The moment generating function of $X$
\begin{equation}
\Phi(t):=\mathbb{E}[\exp(tX)] \label{fgm}
\end{equation}
is supposed to be finite in a non void neighborhood $\mathcal{N}$ of 0. This
hypothesis is usually referred to as a Cram\'er type condition.

The tilted density of $X$ (or Esscher transform of its distribution) with
parameter $t$ in $\mathcal{N}$ is defined on $\mathbb{R}$ by
\[
\pi_{t}(x):=\frac{\exp(tx)}{\Phi(t)}p(x).
\]

For $t\in\mathcal{N}$, the functions
\begin{align}
t  &  \rightarrow m(t):=\frac{d}{dt}\log\Phi(t),\label{1.1}\\
t  &  \rightarrow s^{2}(t):=\frac{d^{2}}{dt^{2}}\log\Phi(t),\label{1.2}\\
t  &  \rightarrow\mu_{j}(t):=\frac{d^{j}}{dt^{j}}\log\Phi(t), j\in(2,\infty).
\label{1.3}
\end{align}
are respectively the expectation and the centered moments of a random variable with density $\pi_{t}$.

When $\Phi$ is steep, meaning that
\begin{equation}
\lim_{t\rightarrow t^{+}}m(t)=\infty\label{1.4}%
\end{equation}
and
\[
\lim_{t\rightarrow t^{-}}m(t)=-\infty
\]
where $t^{+}:=ess\sup\mathcal{N}$ and $t^{-}:=ess\inf\mathcal{N}$ then $m$
parametrizes $\mathbb{R}$ (this is steepness, see Barndorff-Nielsen
\cite{Book2}). We will only require (\ref{1.4}) to hold.

This paper presents sharp approximations for the moments of the tilted density
$\pi_{t}$ under conditions pertaining to the shape of $p$ in its upper tail,
when $t$ tends to the upper bound of $\mathcal{N}$.

Such expansions are relevant in the context of extreme value theory as well as
in approximations of very large deviation probabilities for the empirical mean
of independent and identically distributed summands. We refer to
\cite{MR717931} in the first case, where convergence in type to the Gumbel
extreme distribution follows from the self neglecting property of the function $s^{2}$, and to \cite{MR1284658} in relation with extreme deviation
probabilities. The fact that up to a normalization, and under the natural
regularity conditions assumed in this paper, the tilted distribution with
density $\pi_{t}(x)$ converges to a standard Gaussian law as $t$ tends to the
essential supremum of the set $\mathcal{N}$ is also of some interest.

\bigskip
\section{Notation and hypotheses}
\label{SECTION NOTATION HYP} Thereafter we will use indifferently the notation $f(t)\underset{t\rightarrow
\infty}{\sim}g(t)$ and $f(t)\underset{t\rightarrow\infty}{=}g(t)(1+o(1))$ to
specify that $f$ and $g$ are asymptotically equivalent functions.\newline

The density $p$ is assumed to be of the form
\begin{equation}
p(x)=\exp(-(g(x)-q(x))),\hspace{0.4cm}x\in\mathbb{R}_{+}. \label{2.1}
\end{equation}
For the sake of this paper,
only the form of $p$ for positive $x$ matters.

The function $g$ is positive, convex, four times differentiable and satisfies
\begin{equation}
\frac{g(x)}{x}\underset{x\rightarrow\infty}{\longrightarrow}\infty.
\label{2.2}
\end{equation}
Define
\begin{equation}
h(x):=g^{\prime}(x). \label{h}
\end{equation}
In the present context, due to (\ref{2.2}) and the assumed conditions on $q$
to be stated hereunder, $t^{+}=+\infty$.

Not all positive convex $g$'s satisfying (\ref{2.2}) are adapted to our
purpose. We follow the line of Juszczak and Nagaev \cite{Art9} to describe the assumed regularity conditions of $h$. See also \cite{MR1207549} for
somehow similar conditions.

We firstly assume that the function $h$, which is a positive function defined on $\mathbb{R}_{+}$, is either regularly or rapidly varying
in a neighborhood of infinity; the function $h$ is monotone and, by
(\ref{2.2}), $h(x)\rightarrow\infty$ when $x\rightarrow\infty.$

The following notation is adopted.

$RV(\alpha)$ designates the class of regularly varying functions of index
$\alpha$ defined on $\mathbb{R}_{+}$,

$\psi(t):=h^{\leftarrow}(t)$ designates the inverse of $h.$ Hence $\psi$ is
monotone for large $t$ and $\psi(t)\rightarrow\infty$ when $t\rightarrow\infty$,

$\sigma^{2}(x):=1/h^{\prime}(x)$,

$\hat{x}:=\hat{x}(t)=\psi(t)$,

$\hat{\sigma}:=\sigma(\hat{x})=\sigma(\psi(t))$.

The two cases considered for $h$, the regularly varying case and the rapidly
varying case, are described below. The first one is adapted to regularly
varying functions $g$, whose smoothness is described through the following
condition pertaining to $h$.\newline

\begin{case}
\label{DEF2.1} \textbf{The Regularly varying case.} It will be
assumed that $h$ belongs to the subclass of $RV(\beta)$, $\beta>0$, with
\[
h(x)=x^{\beta}l(x),
\]
where
\begin{equation}
l(x)=c\exp\int_{1}^{x}\frac{\epsilon(u)}{u}du\label{Karamata1}
\end{equation}
for some positive $c$. We assume that $x\mapsto\epsilon(x)$ is twice
differentiable and satisfies
\begin{equation}
\label{2.3}
\left\{
\begin{array}{lll}
&\epsilon(x)\underset{x\rightarrow\infty}{=}o(1),\\
&x|\epsilon^{\prime}(x)|\underset{x\rightarrow\infty}{=}O(1),\\
&x^{2}|\epsilon^{(2)}(x)|\underset{x\rightarrow\infty}{=}O(1).
\end{array}
\right.
\end{equation}
It will also be assumed that
\begin{equation}
|h^{(2)}(x)|\in RV(\theta)\label{h^(2)}
\end{equation}
where $\theta$ is a real number such that $\theta\leq\beta-2$.
\end{case}

\begin{remark}
\label{REM2.1} Under (\ref{Karamata1}), when $\beta\not=1$ then, under (\ref{h^(2)}), $\theta=\beta-2$. Whereas, when $\beta=1$ then $\theta\leq\beta-2$. A sufficient condition for the last assumption (\ref{h^(2)}) is that $\epsilon^{\prime}(t)\in RV(\gamma)$, for some $\gamma<-1$. Also in this case when $\beta=1$, then $\theta=\beta+\gamma-1$.
\end{remark}

\begin{example}
\label{EX2.1} {\textbf{\emph{Weibull density.}}} Let $p$ be a Weibull density
with shape parameter $k>1$ and scale parameter 1, namely
\begin{align*}
p(x) &  =kx^{k-1}\exp(-x^{k}),\hspace{1cm}x\geq0\\
&  =k\exp(-(x^{k}-(k-1)\log x)).
\end{align*}
Take $g(x)=x^{k}-(k-1)\log x$ and $q(x)=0$. Then it holds
\[
h(x)=kx^{k-1}-\frac{k-1}{x}=x^{k-1}\left(  k-\frac{k-1}{x^{k}}\right)  .
\]
Set $l(x)=k-(k-1)/x^{k},x\geq1$, which verifies
\[
l^{\prime}(x)=\frac{k(k-1)}{x^{k+1}}=\frac{l(x)\epsilon(x)}{x}%
\]
with
\[
\epsilon(x)=\frac{k(k-1)}{kx^{k}-(k-1)}.
\]
Since the function $\epsilon(x)$ satisfies the three conditions in
(\ref{2.3}), then $h(x)\in RV(k-1)$.\newline
\end{example}

\begin{case}
\label{DEF2.2} \textbf{The Rapidly varying case.} Here we have
$h^{\leftarrow}(t)=\psi(t)\in RV(0)$ and
\begin{equation}
\psi(t)=c\exp\int_{1}^{t}\frac{\epsilon(u)}{u}du\label{Karamata 2}
\end{equation}
for some positive $c$, and $t\mapsto\epsilon(t)$ is twice differentiable
with
\begin{equation}
\label{2.4}
\left\lbrace
\begin{array}{lll}
&\epsilon(t)\underset{t\rightarrow\infty}{=}o(1),\\
&\frac{t\epsilon^{\prime}(t)}{\epsilon(t)}\underset{t\rightarrow
\infty}{\longrightarrow}0,\\
&\frac{t^{2}\epsilon^{(2)}(t)}{\epsilon(t)}\underset{t\rightarrow
\infty}{\longrightarrow}0.
\end{array}
\right.
\end{equation}
Note that these assumptions imply that $\epsilon(t)\in RV(0).$
\end{case}

\begin{example}
\label{EX2.2} {\textbf{\emph{A rapidly varying density.}}} Define $p$ through
\[
p(x)=c\exp(-e^{x-1}),x\geq0.
\]
Then $g(x)=h(x)=e^{x-1}$ and $q(x)=0$ for all non negative $x$. We show that
$h(x)$ is a rapidly varying function. It holds $\psi(t)=\log t+1$. Since
$\psi^{\prime}(t)=1/t$, let $\epsilon(t)=1/(\log t+1)$ such that $\psi
^{\prime}(t)=\psi(t)\epsilon(t)/t$. Moreover, the three conditions of
(\ref{2.4}) are satisfied. Thus $\psi(t)\in RV(0)$ and $h(x)$ is a rapidly
varying function.
\end{example}

Denote by $\mathcal{R}$ the class of functions with either regular
variation defined as in Case \ref{DEF2.1} or with rapid variation defined as
in Case \ref{DEF2.2}.

We now state hypotheses pertaining to the bounded function $q$ in (\ref{2.1}). We assume that
\begin{equation}
|q(x)|\in RV(\eta),\mbox{ for some $\eta<\theta-\frac{3\beta}{2}-\frac{3}{2}$ if $h\in RV(\beta)$} \label{2.5}
\end{equation}
and
\begin{equation}
|q(\psi(t))|\in RV(\eta),\mbox{ for some $\eta<-\frac{1}{2}$ if $h$ is
rapidly varying.} \label{2.6}
\end{equation}

\bigskip 
\section{An Abelian-type theorem}
\label{SECTION THM CORLL} We have
\begin{theorem}
\label{THM3.1}Let $p(x)$ be defined as in (\ref{2.1}) and $h(x)$ belong to
$\mathcal{R}$. Denote by $m(t)$, $s^{2}(t)$ and $\mu_{j}(t)$ for $j=3,4,...$
the functions defined in (\ref{1.1}), (\ref{1.2}) and (\ref{1.3}). Then it
holds
\begin{align*}
m(t)  &  \underset{t\rightarrow\infty}{=}\psi(t)(1+o(1)),\\
s^{2}(t)  &  \underset{t\rightarrow\infty}{=}\psi^{\prime}(t)(1+o(1)),\\
\mu_{3}(t)  &  \underset{t\rightarrow\infty}{=}\psi^{(2)}(t)(1+o(1)),\\
\mu_{j}(t)  &  \underset{t\rightarrow\infty}{=}\left\{
\begin{array}{ll}
M_{j}s^{j}(t)(1+o(1)), & \mbox{for even $j>3$}\\
\frac{(M_{j+3}-3jM_{j-1})\mu_{3}(t)s^{j-3}(t)}{6}(1+o(1)), & \mbox{for odd
$j>3$}
\end{array}
\right.,
\end{align*}
where $M_{i}$, $i>0$, denotes the $i$th order moment of standard normal distribution.
\end{theorem}

Using (\ref{2.1}), the moment generating function $\Phi(t)$ defined in
(\ref{fgm}) takes on the form
\[
\Phi(t)=\int_{0}^{\infty}e^{tx}p(x)dx=c\int_{0}^{\infty}\exp
(K(x,t)+q(x))dx,\hspace{1cm}t\in(0,\infty)
\]
where
\begin{equation}
K(x,t)=tx-g(x).\label{K}
\end{equation}

If $h\in\mathcal{R}$, then for fixed $t$, $x\mapsto K(x,t)$ is a concave
function and takes its maximum value at $\hat{x}=h^{\leftarrow}(t)$.

As a direct by-product of Theorem \ref{THM3.1} we obtain the following Abel
type result.

\begin{theorem}
\label{THM3.2} Under the same hypotheses as in Theorem \ref{THM3.1}, we have
\[
\Phi(t)=\sqrt{2\pi}\hat{\sigma}e^{K(\hat{x},t)}(1+o(1)).
\]
\end{theorem}

\begin{remark}
It is easily verified that this result is in accordance with Theorem 4.12.11
of \cite{Book1}, Theorem 3 of \cite{Art1} and Theorem 4.2 of \cite{Art9}.
Some classical consequence of Kasahara's Tauberian theorem can be paralleled with Theorem \ref{THM3.2}. Following Theorem 4.2.10 in \cite{Book1}, with $f$ defined as $g$ above, it follows that $-\log\int_x^\infty p(v)dv\sim g(x)$ as $x\rightarrow\infty$ under Case \ref{DEF2.1}, a stronger assumption than required in Theorem 4.2.10 of \cite{Book1}. Theorem 4.12.7 in \cite{Book1} hence applies and provides an asymptotic equivalent for $\log\Phi(t)$ as $t\rightarrow\infty$; Theorem \ref{THM3.2} improves on this result, at the cost of the additional regularity assumptions of Case \ref{DEF2.1}. Furthermore, no result of this kind seems to exist in Case \ref{DEF2.2}.
\end{remark}

We also derive the following consequence of Theorem \ref{THM3.1}.

\begin{theorem}
\label{Thm3.4} Under the present hypotheses, denote $\mathcal{X}_{t}$ a random
variable with density $\pi_{t}(x)$. Then as $t\rightarrow\infty$, the family
of random variables
\[
\frac{\mathcal{X}_{t}-m(t)}{s(t)}%
\]
converges in distribution to a standard normal distribution.
\end{theorem}

\begin{remark}
This result holds under various hypotheses, as developped for example in
\cite{MR1207549} or \cite{MR717931}. Under log-concavity of $p$ it also holds
locally; namely the family of densities  $\pi_{t}$  converges pointwise to the standard gaussian density; this yields asymptotic results for the extreme
deviations of the empirical mean of i.i.d. summands with light tails (see
\cite{MR1284658}), and also provides sufficient conditions for $P_{X}$ to
belong to the domain of attraction of the Gumbel distribution for the maximum, through criterions pertaining to the Mill's ratio (see \cite{MR717931}).
\end{remark}

\begin{remark}
That $g$ is four times derivable can be relaxed; in Case \ref{DEF2.1} with $\beta>2$ or in Case \ref{DEF2.2}, $g$ a three times derivable function, together with the two first lines in (\ref{2.3}) and (\ref{2.4}), provides Theorems \ref{THM3.1}, \ref{THM3.2} and \ref{Thm3.4}. Also it may be seen that the order of differentiability of $g$ in Case \ref{DEF2.1} with $0<\beta\leq 2$ is related to the order of the moment of the tilted distribution for which an asymptotic equivalent is obtained. This will be developed in a forthcoming paper.
\end{remark}

The proofs of the above results rely on Lemmas \ref{LEM4.1} to \ref{LEM4.5}. Lemma \ref{LEM4.1} is
instrumental for Lemma \ref{LEM4.5}.

\bigskip
\section{Appendix: Proofs}

\label{SECTION PROOFS}The following Lemma provides a simple argument for the local uniform convergence of regularly varying functions.

\begin{lemma}
\label{LEM4.0}
Consider $l(t)\in RV(\alpha), \alpha\in\mathbb{R}$. For any function $f$ such that $f(t)\underset{t\rightarrow\infty}{=}o(t)$, it holds
\begin{equation}
\sup_{\vert x\vert\leq f(t)}\vert l(t+x)\vert\underset{t\rightarrow\infty}{\sim}\vert l(t)\vert.\label{LUC1}
\end{equation}
If $f(t)=at$ with $0<a<1$, then it holds
\begin{equation}
\sup_{\vert x\vert\leq at}\vert l(t+x)\vert\underset{t\rightarrow\infty}{\sim}(1+a)^\alpha\vert l(t)\vert.\label{LUC2}
\end{equation}
\end{lemma}

\begin{proof}
By Theorem 1.5.2 of \cite{Book1}, if $l(t)\in RV(\alpha)$, then for all $I$
\[
\sup_{\lambda\in I}\left\vert\frac{l(\lambda t)}{l(t)}-\lambda^\alpha\right\vert\underset{t\rightarrow\infty}{\longrightarrow}0,
\]
with $I=\lbrack A, B\rbrack \ (0<A\leq B<\infty)$ if $\alpha=0$, $I=(0, B\rbrack \ (0<B<\infty)$ if $\alpha>0$ and $I=\lbrack A, \infty) \ (0<A<\infty)$ if $\alpha<0$.

Putting $\lambda=1+x/t$ with $f(t)\underset{t\rightarrow\infty}{=}o(t)$, we obtain
\[
\sup_{\vert x\vert\leq f(t)}\left\vert\frac{l(t+x)}{l(t)}\right\vert-\left(1+\frac{f(t)}{t}\right)^\alpha\underset{t\rightarrow\infty}{\longrightarrow}0,
\]
which implies (\ref{LUC1}).

When $f(t)=at$ with $0<a<1$, we get
\[
\sup_{\vert x\vert\leq at}\left\vert\frac{l(t+x)}{l(t)}\right\vert-(1+a)^\alpha\underset{t\rightarrow\infty}{\longrightarrow}0,
\]
which implies (\ref{LUC2}).\\
\end{proof}

Now we quote some simple expansions pertaining to the
function $h$ under the two cases considered in the above Section
\ref{SECTION NOTATION HYP}.

\begin{lemma}
\label{LEM2.1} We have under Case \ref{DEF2.1},
\begin{align*}
h^{\prime}(x)  &  =\frac{h(x)}{x}[\beta+\epsilon(x)],\\
h^{(2)}(x)  &  =\frac{h(x)}{x^{2}}[\beta(\beta-1)+a\epsilon(x)+\epsilon
^{2}(x)+x\epsilon^{\prime}(x)],\\
h^{(3)}(x)  &  =\frac{h(x)}{x^{3}}[\beta(\beta-1)(\beta-2)+b\epsilon
(x)+c\epsilon^{2}(x)+\epsilon^{3}(x)\\
&  \hspace{1.5cm}+x\epsilon^{\prime}(x)(d+e\epsilon(x))+x^{2}\epsilon
^{(2)}(x)].
\end{align*}
where $a,b,c,d,e$ are some real constants.
\end{lemma}

\begin{corollary}
\label{COR2.1} We have under Case \ref{DEF2.1}, $h^{\prime}(x)\underset{x\rightarrow\infty}{\sim}\beta h(x)/x$ and $|h^{(i)}(x)|\leq C_{i}h(x)/x^{i},i=1,2,3$, for
some constants $C_{i}$ and for large $x$ .
\end{corollary}

\begin{corollary}
\label{COR2.2} We have under Case \ref{DEF2.1}, $\hat{x}(t)=\psi(t)\in RV(1/\beta)$ (see Theorem (1.5.15) of \cite{Book1}) and $\hat{\sigma}^{2}(t)=\psi^{\prime}(t)\sim\beta^{-1}\psi(t)/t\in RV(1/\beta-1)$.
\end{corollary}

It also holds

\begin{lemma}
\label{LEM2.2} We have under Case \ref{DEF2.2},
\[
\psi^{(2)}(t)\underset{t\rightarrow\infty}{\sim}-\frac{\psi(t)\epsilon
(t)}{t^{2}}\ and\ \psi^{(3)}(t)\underset{t\rightarrow\infty}{\sim}2\frac
{\psi(t)\epsilon(t)}{t^{3}}.
\]
\end{lemma}

\begin{lemma}
\label{LEM2.3} We have under Case \ref{DEF2.2},
\begin{align*}
h^{\prime}(\psi(t))  &  =\frac{1}{\psi^{\prime}(t)}=\frac{t}{\psi
(t)\epsilon(t)},\\
h^{(2)}(\psi(t))  &  =-\frac{\psi^{(2)}(t)}{(\psi^{\prime}(t))^{3}}\underset{t\rightarrow\infty}{\sim}\frac{t}{\psi^{2}(t)\epsilon^{2}(t)},\\
h^{(3)}(\psi(t))  &  =\frac{3(\psi^{(2)}(t))^{2}-\psi^{(3)}(t)\psi^{\prime
}(t)}{(\psi^{\prime}(t))^{5}}\underset{t\rightarrow\infty}{\sim}\frac{t}{\psi^{3}(t)\epsilon^{3}(t)}.
\end{align*}
\end{lemma}

\begin{corollary}
\label{COR2.3} We have under Case \ref{DEF2.2}, $\hat{x}(t)=\psi(t)\in RV(0)$ and $\hat{\sigma}^{2}(t)=\psi^{\prime}(t)=\psi(t)\epsilon(t)/t\in RV(-1)$.
Moreover, we have $h^{(i)}(\psi(t))\in RV(1),i=1,2,3$.
\end{corollary}

Before beginning the proofs of our results we quote that the regularity
conditions (\ref{Karamata1}) and (\ref{Karamata 2}) pertaining to the
function $h$ allow for the above simple expansions. Substituting the constant
$c$ in (\ref{Karamata1}) and (\ref{Karamata 2}) by functions $x\rightarrow
c(x)$ which converge smoothly to some positive constant $c$ adds noticeable complexity.

We now come to the proofs of five Lemmas which provide the asymptotics leading to Theorem \ref{THM3.1} and Theorem \ref{THM3.2}.

\begin{lemma}
\label{LEM4.1}It holds
\[
\frac{\log\hat{\sigma}}{\int_{1}^{t}\psi(u)du}\underset{t\rightarrow
\infty}{\longrightarrow}0.
\]
\end{lemma}

\begin{proof}
By Corollaries \ref{COR2.2} and \ref{COR2.3}, we have that $\psi(t)\in RV(1/\beta)$ in Case \ref{DEF2.1} and $\psi(t)\in RV(0)$ in Case \ref{DEF2.2}. Using Theorem 1 of \cite{Book3}, Chapter 8.9 or Proposition 1.5.8 of \cite{Book1}, we obtain
\begin{equation}
\label{INT-psi}
\int_{1}^{t}\psi(u)du\underset{t\rightarrow\infty}{\sim}
\left\{
\begin{array}{ll}
t\psi(t)/(1+1/\beta)\in RV(1+1/\beta) &\mbox{ if $h\in RV(\beta)$}\\
t\psi(t)\in RV(1) &\mbox{ if $h$ is rapidly varying}	
\end{array}
\right..
\end{equation}
Also by Corollaries \ref{COR2.2} and \ref{COR2.3}, we have that $\hat{\sigma}^2\in RV(1/\beta-1)$ in Case \ref{DEF2.1} and $\hat{\sigma}^2\in RV(-1)$ in Case \ref{DEF2.2}. Thus $t\mapsto\log\hat{\sigma}\in RV(0)$ by composition and
\[
\frac{\log\hat{\sigma}}{\int_{1}^{t}\psi(u)du}\underset{t\rightarrow\infty}{\sim}
\left\{
\begin{array}{ll}
\frac{\beta+1}{\beta}\times\frac{\log\hat{\sigma}}{t\psi(t)}\in RV\left(-1-\frac{1}{\beta}\right) &\mbox{ if $h\in RV(\beta)$}\\
\frac{\log\hat{\sigma}}{t\psi(t)}\in RV(-1) &\mbox{ if $h$ is rapidly varying}
\end{array}
\right.,
\]
which proves the claim.\\
\end{proof}

The next steps of the proof make use of the function
\[
L(t):=\left(\log t\right)^{3}.
\]

\begin{lemma}
\label{LEM4.2} We have
\[
\sup_{|x|\leq\hat{\sigma}L(t)}\left\vert \frac{h^{(3)}(\hat{x}+x)}%
{h^{(2)}(\hat{x})}\right\vert \hat{\sigma}L^{4}(t)\underset{t\rightarrow
\infty}{\longrightarrow}0.
\]

\end{lemma}

\begin{proof}
\textit{Case 1}. By Corollary \ref{COR2.1} and
by (\ref{h^(2)}) we have
\[
|h^{(3)}(x)|\leq C\frac{|h^{(2)}(x)|}{x},
\]
for some constant $C$ and $x$ large. Since, by Corollary \ref{COR2.2},
$\hat{x}\in RV(1/\beta)$ and $\hat{\sigma}^{2}\in RV(1/\beta-1)$, we have
\[
\frac{|x|}{\hat{x}}\leq\frac{\hat{\sigma}L(t)}{\hat{x}}\in RV\left(  -\frac
{1}{2}-\frac{1}{2\beta}\right)
\]
and $|x|/\hat{x}\underset{t\rightarrow\infty}{\longrightarrow}0$ uniformly in
$\{x:|x|\leq\hat{\sigma}L(t)\}$. For large $t$ and all $x$ such that
$|x|\leq\hat{\sigma}L(t)$, we have
\[
|h^{(3)}(\hat{x}+x)|\leq C\frac{|h^{(2)}(\hat{x}+x)|}{\hat{x}+x}\leq
C\sup_{|x|\leq\hat{\sigma}L(t)}\frac{|h^{(2)}(\hat{x}+x)|}{\hat{x}+x}%
\]
whence
\[
\sup_{|x|\leq\hat{\sigma}L(t)}|h^{(3)}(\hat{x}+x)|\leq C\sup_{|x|\leq
\hat{\sigma}L(t)}\frac{|h^{(2)}(\hat{x}+x)|}{\hat{x}+x}%
\]
where
\[
\sup_{|x|\leq\hat{\sigma}L(t)}\frac{|h^{(2)}(\hat{x}+x)|}{\hat{x}%
+x}\underset{t\rightarrow\infty}{\sim}\frac{|h^{(2)}(\hat{x})|}{\hat{x}},
\]
using (\ref{LUC1}) for the regularly varying function
$|h^{(2)}(\hat{x})|\in RV(\theta/\beta)$, with $f(t)=\hat{\sigma}L(t)\underset{t\rightarrow\infty}{=}o(\hat{x})$. Thus for $t$ large enough and for all $\delta>0$
\[
\sup_{|x|\leq\hat{\sigma}L^{4}(t)}\left\vert\frac{h^{(3)}(\hat{x}+x)}
{h^{(2)}(\hat{x})}\right\vert \hat{\sigma}L^{4}(t)\leq C\frac{\hat{\sigma
}L^{4}(t)}{\hat{x}}(1+\delta)\in RV\left( \frac{1}{2\beta}-\frac{1}{2}-\frac
{1}{\beta}\right)  ,
\]
which proves Lemma \ref{LEM4.2} in Case \ref{DEF2.1}.\newline

\textit{Case 2}. By Lemma \ref{LEM2.3}, we have
that $h^{(3)}(\psi(t))\in RV(1)$. Using (\ref{LUC2}), we have for $0<a<1$ and $t$ large enough
\[
\sup_{|v|\leq at}|h^{(3)}(\psi(t+v))|\underset{t\rightarrow\infty}{\sim
}(1+a)h^{(3)}(\psi(t)).
\]

In the present case $\hat{x}\in RV(0)$ and $\hat{\sigma}^{2}\in RV(-1)$.
Setting $\psi(t+v)=\hat{x}+x=\psi(t)+x$, we have $x=\psi(t+v)-\psi(t)$ and
$A:=\psi(t-at)-\psi(t)\leq x\leq\psi(t+at)-\psi(t)=:B$, since $t\mapsto\psi(t)$ is an increasing function. It follows that
\[
\sup_{|v|\leq at}h^{(3)}(\psi(t+v))=\sup_{A\leq x\leq B}h^{(3)}(\hat{x}+x).
\]
Now note that (cf. page 127 in \cite{Book1})
\[
B=\psi(t+at)-\psi(t)=\int_{t}^{t+at}\psi^{\prime}(z)dz=\int_{t}^{t+at}%
\frac{\psi(z)\epsilon(z)}{z}dz\underset{t\rightarrow\infty}{\sim}%
\psi(t)\epsilon(t)\log(1+a),
\]
since $\psi(t)\epsilon(t)\in RV(0)$. Moreover, we have
\[
\frac{\hat{\sigma}L(t)}{\psi(t)\epsilon(t)}\in RV(-1)\mbox{ and }\frac
{\hat{\sigma}L(t)}{\psi(t)\epsilon(t)}\underset{t\rightarrow\infty
}{\longrightarrow}0.
\]
It follows that $\hat{\sigma}L(t)\underset{t\rightarrow\infty}{=}o(B)$ and in
a similar way, we have $\hat{\sigma}L(t)\underset{t\rightarrow\infty}{=}o(A)$.
Using Lemma \ref{LEM2.3} and since $\hat{\sigma}L^{4}(t)\in RV(-1/2)$, it
follows that for $t$ large enough and for all $\delta>0$
\begin{align*}
\sup_{|x|\leq\hat{\sigma}L(t)}\frac{|h^{(3)}(\psi(t+v))|}{|h^{(2)}(\psi
(t))|}\hat{\sigma}L^{4}(t)  &  \leq\sup_{A\leq x\leq B}\frac{|h^{(3)}%
(\psi(t+v))|}{|h^{(2)}(\psi(t))|}\hat{\sigma}L^{4}(t)\\
&  \leq(1+a)\frac{\hat{\sigma}L^{4}(t)}{\psi(t)\epsilon(t)}(1+\delta)\in
RV\left(  -\frac{1}{2}\right)  ,
\end{align*}
which concludes the proof of Lemma \ref{LEM4.2} in Case \ref{DEF2.2}.\newline
\end{proof}

\begin{lemma}
\label{LEM4.3} We have
\begin{align*}
&  |h^{(2)}(\hat{x})|\hat{\sigma}^{4}\underset{t\rightarrow\infty
}{\longrightarrow}0,\\
&  |h^{(2)}(\hat{x})|\hat{\sigma}^{3}L(t)\underset{t\rightarrow\infty
}{\longrightarrow}0.
\end{align*}
\end{lemma}

\begin{proof}
\textit{Case 1}. By Corollary \ref{COR2.1} and Corollary \ref{COR2.2}, we have
\[
|h^{(2)}(\hat{x})|\hat{\sigma}^{4}\leq\frac{C_{2}}{\beta^{2}t}\in RV(-1)
\]
and
\[
|h^{(2)}(\hat{x})|\hat{\sigma}^{3}L(t)\leq\frac{C_{2}}{\beta^{3/2}}\frac
{L(t)}{\sqrt{t\psi(t)}}\in RV\left(  -\frac{1}{2\beta}-\frac{1}{2}\right)  ,
\]
proving the claim.\newline

\textit{Case 2}. We have by Lemma \ref{LEM2.3}
and Corollary \ref{COR2.3}
\[
h^{(2)}(\hat{x})\hat{\sigma}^{4}\underset{t\rightarrow\infty}{\sim}\frac{1}%
{t}\in RV(-1)
\]
and
\[
h^{(2)}(\hat{x})\hat{\sigma}^{3}L(t)\underset{t\rightarrow\infty}{\sim}%
\frac{L(t)}{\sqrt{t\psi(t)\epsilon(t)}}\in RV\left(  -\frac{1}{2}\right)  ,
\]
which concludes the proof of Lemma \ref{LEM4.3}.
\end{proof}

We now define some functions to be used in the sequel.  A Taylor-Lagrange
expansion of $K(x,t)$ in a neighborhood of $\hat{x}$ yields
\begin{equation}
K(x,t)=K(\hat{x},t)-\frac{1}{2}h^{\prime}(\hat{x})(x-\hat{x})^{2}-\frac{1}%
{6}h^{(2)}(\hat{x})(x-\hat{x})^{3}+\varepsilon(x,t),\label{K-decomp}
\end{equation}
where, for some $\theta\in(0,1)$,
\begin{equation}
\varepsilon(x,t)=-\frac{1}{24}h^{(3)}(\hat{x}+\theta(x-\hat{x}))(x-\hat
{x})^{4}.\newline\label{eps}
\end{equation}

\begin{lemma}
\label{LEM4.4} We have
\[
\sup_{y\in\lbrack-L(t),L(t)]}\frac{|\xi(\hat{\sigma}y+\hat{x},t)|}{h^{(2)}(\hat{x})\hat{\sigma}^{3}}\underset{t\rightarrow\infty
}{\longrightarrow}0,
\]
where $\xi(x,t)=\varepsilon(x,t)+q(x)$ and $\varepsilon(x,t)$ is defined in
(\ref{eps}).\newline
\end{lemma}

\begin{proof}
For $y\in\lbrack-L(t),L(t)]$, by \ref{eps}, it holds
\[
\left\vert \frac{\varepsilon(\hat{\sigma}y+\hat{x},t)}{h^{(2)}(\hat{x})\hat{\sigma}^{3}}\right\vert \leq\left\vert \frac{h^{(3)}(\hat{x}+\theta\hat{\sigma}y)(\hat{\sigma}y)^{4}}{h^{(2)}(\hat{x})\hat{\sigma}^{3}}\right\vert \leq\left\vert \frac{h^{(3)}(\hat{x}+\theta\hat{\sigma}y)\hat{\sigma}L^{4}(t)}{h^{(2)}(\hat{x})}\right\vert ,
\]
with $\theta\in(0,1)$. Let $x=\theta\hat{\sigma}y$. It then holds $|x|\leq
\hat{\sigma}L(t)$. Therefore by Lemma \ref{LEM4.2}
\[
\sup_{y\in\lbrack-L(t),L(t)]}\left\vert\frac{\varepsilon(\hat{\sigma}y+\hat{x},t)|}{h^{(2)}(\hat{x})\hat{\sigma}^{3}}\right\vert \leq\sup_{|x|\leq\hat{\sigma}L(t)}\left\vert\frac{h^{(3)}(\hat{x}+x)}{h^{(2)}(\hat{x})}\hat{\sigma}L^{4}(t)\right\vert\underset{t\rightarrow\infty}{\longrightarrow}0.
\]
It remains to prove that
\begin{equation}
\sup_{y\in\lbrack-L(t),L(t)]}\left\vert \frac{q(\hat{\sigma}y+\hat{x})}{h^{(2)}(\hat{x})\hat{\sigma}^{3}}\right\vert \underset{t\rightarrow
\infty}{\longrightarrow}0.\label{demLemm4.4}
\end{equation}

\textit{Case 1}. By (\ref{h^(2)}) and by composition, $|h^{(2)}(\hat{x})|\in RV(\theta/\beta)$. Using Corollary \ref{COR2.1} we obtain
\[
|h^{(2)}(\hat{x})\hat{\sigma}^{3}|\underset{t\rightarrow\infty}{\sim}\frac{|h^{(2)}(\hat{x})|\psi^{3/2}(t)}{\beta^{3/2}t^{3/2}}\in RV\left(\frac{\theta
}{\beta}+\frac{3}{2\beta}-\frac{3}{2}\right).
\]

Since, by (\ref{2.5}), $|q(\hat{x})|\in RV(\eta/\beta)$, for $\eta
<\theta-3\beta/2+3/2$ and putting $x=\hat{\sigma}y$, we obtain
\begin{align*}
\sup_{y\in\lbrack-L(t),L(t)]}\left\vert \frac{q(\hat{\sigma}y+\hat{x}
)}{h^{(2)}(\hat{x})\hat{\sigma}^{3}}\right\vert  &  =\sup_{|x|\leq\hat{\sigma
}L(t)}\left\vert \frac{q(\hat{x}+x)}{h^{(2)}(\hat{x})\hat{\sigma}^{3}
}\right\vert \\
&  \underset{t\rightarrow\infty}{\sim}\frac{|q(\hat{x})|}{|h^{(2)}(\hat
{x})\hat{\sigma}^{3}|}\in RV\left(  \frac{\eta-\theta}{\beta}-\frac{3}{2\beta
}+\frac{3}{2}\right),
\end{align*}
which proves (\ref{demLemm4.4}).\newline

\textit{Case 2}. By Lemma \ref{LEM2.3} and Corollary \ref{COR2.3}, we have
\[
|h^{(2)}(\hat{x})\hat{\sigma}^{3}|\underset{t\rightarrow\infty}{\sim}\frac
{1}{\sqrt{t\psi(t)\epsilon(t)}}\in RV\left(-\frac{1}{2}\right).
\]

As in Lemma \ref{LEM4.2}, since by (\ref{2.6}), $q(\psi(t))\in RV(\eta)$, then
we obtain, with $\eta<-1/2$
\begin{align*}
\sup_{y\in\lbrack-L(t),L(t)]}\left\vert \frac{q(\hat{\sigma}y+\hat{x}%
)}{h^{(2)}(\hat{x})\hat{\sigma}^{3}}\right\vert  &  =\sup_{|x|\leq\hat{\sigma
}L(t)}\left\vert \frac{q(\hat{x}+x)}{h^{(2)}(\hat{x})\hat{\sigma}^{3}%
}\right\vert \\
&  \leq\sup_{|v|\leq at}\left\vert \frac{q(\psi(t+v))}{h^{(2)}(\hat{x}%
)\hat{\sigma}^{3}}\right\vert \\
&  \leq(1+a)^{\eta}q(\psi(t))\sqrt{t\psi(t)\epsilon(t)}(1+\delta)\in RV\left(
\eta+\frac{1}{2}\right)  ,
\end{align*}
for all $\delta>0$, with $a<1$, $t$ large enough and $\eta+1/2<0$. This proves (\ref{demLemm4.4}).\\
\end{proof}

\begin{lemma}
\label{LEM4.5}For $\alpha\in\mathbb{N}$, denote
\[
\Psi(t,\alpha):=\int_{0}^{\infty}(x-\hat{x})^{\alpha}e^{tx}p(x)dx.
\]
Then
\[
\Psi(t,\alpha)\underset{t\rightarrow\infty}{=}\hat{\sigma}^{\alpha
+1}e^{K(\hat{x},t)}T_{1}(t,\alpha)(1+o(1)),
\]
where
\begin{equation}
T_{1}(t,\alpha)=\int_{-\frac{L^{1/3}(t)}{\sqrt{2}}}^{\frac{L^{1/3}(t)}%
{\sqrt{2}}}y^{\alpha}\exp(-\frac{y^{2}}{2})dy-\frac{h^{(2)}(\hat{x}%
)\hat{\sigma}^{3}}{6}\int_{-\frac{L^{1/3}(t)}{\sqrt{2}}}^{\frac{L^{1/3}%
(t)}{\sqrt{2}}}y^{3+\alpha}\exp(-\frac{y^{2}}{2})dy. \label{4.1}%
\end{equation}

\end{lemma}

\begin{proof}
We define the interval $I_{t}$ as follows
\[
I_{t}:=\left(  -\frac{L^{\frac{1}{3}}(t)\hat{\sigma}}{\sqrt{2}},\frac
{L^{\frac{1}{3}}(t)\hat{\sigma}}{\sqrt{2}}\right)  .
\]

For large enough $\tau$, when $t\rightarrow\infty$ we can partition
$\mathbb{R}_{+}$ into
\[
\mathbb{R}_{+}=\{x:0<x<\tau\}\cup\{x:x\in\hat{x}+I_{t}\}\cup\{x:x\geq
\tau,x\not \in \hat{x}+I_{t}\},
\]
where for $x>\tau$, $q(x)<\log2$. Thus we have
\begin{equation}
p(x)<2e^{-g(x)}. \label{4.2}
\end{equation}
For fixed $\tau$, $\{x:0<x<\tau\}\cap\{x:x\in\hat{x}+I_{t}%
\}=\emptyset$. Therefore $\tau<\hat{x}-\frac{L^{\frac{1}{3}}(t)\hat{\sigma}}{\sqrt{2}}\leq\hat{x}$ for $t$ large enough. Hence it holds
\begin{equation}
\Psi(t,\alpha)=:\Psi_{1}(t,\alpha)+\Psi_{2}(t,\alpha)+\Psi_{3}(t,\alpha),
\label{4.3}
\end{equation}
where
\begin{align*}
\Psi_{1}(t,\alpha)  &  =\int_{0}^{\tau}(x-\hat{x})^{\alpha}e^{tx}p(x)dx,\\
\Psi_{2}(t,\alpha)  &  =\int_{x\in\hat{x}+I_{t}}(x-\hat{x})^{\alpha}
e^{tx}p(x)dx,\\
\Psi_{3}(t,\alpha)  &  =\int_{x\not \in \hat{x}+I_{t},x\geq\tau}(x-\hat
{x})^{\alpha}e^{tx}p(x)dx.
\end{align*}
We estimate $\Psi_{1}(t,\alpha)$, $\Psi_{2}(t,\alpha)$ and $\Psi_{3}%
(t,\alpha)$ in Step 1, Step 2 and Step 3.\newline

\textit{Step 1:} Since $q$ is bounded, we consider
\[
\log d=\sup_{x\in(0, \tau)}q(x)
\]
and for $t$ large enough, we have
\[
\left\vert\Psi_{1}(t,\alpha)\right\vert \leq\int_{0}^{\tau}\left\vert
x-\hat{x}\right\vert ^{\alpha}e^{tx}p(x)dx\leq d\int_{0}^{\tau}\hat
{x}^{\alpha}e^{tx}dx,
\]
since when $0<x<\tau<\hat{x}$ then $\left\vert x-\hat{x}\right\vert =\hat
{x}-x<\hat{x}$ for $t$ large enough and $g$ is positive.\newline Since, for
$t$ large enough, we have
\[
\int_{0}^{\tau}\hat{x}^{\alpha}e^{tx}dx=\hat{x}^{\alpha}\frac{e^{t\tau}}
{t}-\frac{\hat{x}^{\alpha}}{t}\leq\hat{x}^{\alpha}\frac{e^{t\tau}}{t},
\]
we obtain
\begin{equation}
\left\vert\Psi_{1}(t,\alpha)\right\vert\leq d\hat{x}^{\alpha}\frac{e^{t\tau
}}{t}. \label{4.4}
\end{equation}

We now show that for $h\in\mathcal{R}$, it holds
\begin{equation}
\hat{x}^{\alpha}\frac{e^{t\tau}}{t}\underset{t\rightarrow\infty}{=}o(\vert
\hat{\sigma}^{\alpha+1}\vert e^{K(\hat{x},t)}\vert h^{(2)}(\hat{x})\hat
{\sigma}^{3}\vert), \label{4.5}
\end{equation}
with $K(x,t)$ defined as in (\ref{K}). This is equivalent to
\[
\frac{\hat{x}^{\alpha}e^{t\tau}}{t\vert\hat{\sigma}^{\alpha+4}h^{(2)}(\hat
{x})\vert}\underset{t\rightarrow\infty}{=}o(e^{K(\hat{x},t)}),
\]
which is implied by
\begin{equation}
-(\alpha+4)\log\vert\hat{\sigma}\vert-\log t+\alpha\log\hat{x}+\tau
t-\log\vert h^{(2)}(\hat{x})\vert\underset{t\to\infty}{=}o(K(\hat{x},t)),
\label{4.6}%
\end{equation}
if $K(\hat{x},t)\underset{t\rightarrow\infty}{\longrightarrow}\infty$.

Setting $u=h(v)$ in $\int_{1}^{t}\psi(u)du$, we have
\[
\int_{1}^{t}\psi(u)du=t\psi(t)-\psi(1)-g(\psi(t))+g(\psi(1)).
\]
Since $K(\hat{x},t)=t\psi(t)-g(\psi(t))$, we obtain
\begin{equation}
K(\hat{x},t)=\int_{1}^{t}\psi(u)du+\psi(1)-g(\psi(1))\underset{t\rightarrow
\infty}{\sim}\int_{1}^{t}\psi(u)du. \label{4.7}
\end{equation}

Let us denote (\ref{INT-psi}) by
\begin{equation}
K(\hat{x},t)\underset{t\rightarrow\infty}{\sim}at\psi(t), \label{4.9}
\end{equation}
with
\[
a=\left\{
\begin{array}
[c]{ll}%
\frac{\beta}{\beta+1} & \mbox{ if $h\in RV(\beta)$}\\
1 & \mbox{ if $h$ is rapidly varying}
\end{array}
\right..
\]
We have to show that each term in (\ref{4.6}) is $o(K(\hat{x},t))$.

\begin{enumerate}
\item[1.] By Lemma \ref{LEM4.1}, $\log\hat{\sigma}\underset{t\rightarrow\infty
}{=}o(\int_{1}^{t}\psi(u)du)$. Hence $\log\hat{\sigma}\underset{t\rightarrow
\infty}{=}o(K(\hat{x},t))$.

\item[2.] By Corollary \ref{COR2.2} and Corollary \ref{COR2.3}, we have
\[
\frac{t}{K(\hat{x},t)}\underset{t\rightarrow\infty}{\sim}\frac{1}{a\psi
(t)}\underset{t\rightarrow\infty}{\longrightarrow}0.
\]
Thus $t\underset{t\rightarrow\infty}{=}o(K(\hat{x},t))$.

\item[3.] Since $\hat{x}=\psi(t)\underset{t\rightarrow\infty}{\longrightarrow
}\infty$, it holds
\[
\left\vert \frac{\log\hat{x}}{K(\hat{x},t)}\right\vert \leq C\frac{\psi
(t)}{K(\hat{x},t)},
\]
for some positive constant $C$ and $t$ large enough. Moreover by (\ref{4.9}),
we have
\[
\frac{\psi(t)}{K(\hat{x},t)}\underset{t\rightarrow\infty}{\sim}\frac{1}%
{at}\underset{t\rightarrow\infty}{\longrightarrow}0.
\]
Hence $\log\hat{x}\underset{t\rightarrow\infty}{=}o(K(\hat{x},t))$.

\item[4.] Using (\ref{4.9}), $\log\vert h^{(2)}(\hat{x}
)\vert\in RV(0)$ and $\log\vert h^{(2)}(\hat{x})\vert\underset{t\rightarrow\infty
}{=}o(K(\hat{x},t))$.

\item[5.] Since $\log t\underset{t\rightarrow\infty}{=}o(t)$ and
$t\underset{t\rightarrow\infty}{=}o(K(\hat{x},t))$, we obtain $\log
t\underset{t\rightarrow\infty}{=}o(K(\hat{x},t))$.
\end{enumerate}

Since (\ref{4.6}) holds and $K(\hat{x},t)\underset{t\rightarrow\infty
}{\longrightarrow}\infty$ by (\ref{4.7}) and (\ref{4.9}), we then get
(\ref{4.5}).\newline

(\ref{4.4}) and (\ref{4.5}) yield together
\begin{equation}
\left\vert \Psi_{1}(t,\alpha)\right\vert \underset{t\rightarrow\infty
}{=}o(\vert\hat{\sigma}^{\alpha+1}\vert e^{K(\hat{x},t)}\vert h^{(2)}(\hat
{x})\hat{\sigma}^{3}\vert). \label{4.10}%
\end{equation}

When $\alpha$ is even,
\begin{equation}
T_{1}(t,\alpha)=\int_{-\frac{t^{1/3}}{\sqrt{2}}}^{\frac{t^{1/3}}{\sqrt{2}}%
}y^{\alpha}\exp(-\frac{y^{2}}{2})dy\underset{t\rightarrow\infty}{\sim}%
\sqrt{2\pi}M_{\alpha},\label{T1pair}
\end{equation}
where $M_{\alpha}$ is the moment of order $\alpha$ of a standard normal
distribution. Thus by Lemma \ref{LEM4.3} we have
\begin{equation}
\frac{h^{(2)}(\hat{x})\hat{\sigma}^{3}}{T_{1}(t,\alpha)}\underset{t\rightarrow\infty}{\longrightarrow}0.\label{resT1pair}
\end{equation}

When $\alpha$ is odd,
\begin{equation}
T_{1}(t,\alpha)=-\frac{h^{(2)}(\hat{x})\hat{\sigma}^{3}}{6}\int_{-\frac
{l^{\frac{1}{3}}}{\sqrt{2}}}^{\frac{l^{\frac{1}{3}}}{\sqrt{2}}}y^{3+\alpha
}\exp\left(-\frac{y^{2}}{2}\right)dy\underset{t\rightarrow\infty}{\sim
}-\frac{h^{(2)}(\hat{x})\hat{\sigma}^{3}}{6}\sqrt{2\pi}M_{\alpha+3},\label{T1impair}
\end{equation}
where $M_{\alpha+3}$ is the moment of order $\alpha+3$ of a standard normal
distribution. Thus we have
\begin{equation}
\frac{h^{(2)}(\hat{x})\hat{\sigma}^{3}}{T_{1}(t,\alpha)}\underset{t\rightarrow
\infty}{\sim}-\frac{6}{\sqrt{2\pi}M_{\alpha+3}}.\label{resT1impair}
\end{equation}

Combined with (\ref{4.10}), (\ref{resT1pair}) and (\ref{resT1impair}) imply for $\alpha\in\mathbb{N}$
\begin{equation}
\left\vert \Psi_{1}(t,\alpha)\right\vert \underset{t\rightarrow\infty
}{=}o(\hat{\sigma}^{\alpha+1}e^{K(\hat{x},t)}T_{1}(t,\alpha)). \label{4.11}%
\end{equation}

\vspace{0.4cm}

\textit{Step 2:} By (\ref{2.1}) and (\ref{K-decomp})
\begin{align*}
\Psi_{2}(t,\alpha)  &  =\int_{x\in\hat{x}+I_{t}}(x-\hat{x})^{\alpha
}e^{K(x,t)+q(x)}dx\\
&  =\int_{x\in\hat{x}+I_{t}}(x-\hat{x})^{\alpha}e^{K(\hat{x},t)-\frac{1}
{2}h^{\prime}(\hat{x})(x-\hat{x})^{2}-\frac{1}{6}h^{(2)}(\hat{x})(x-\hat
{x})^{3}+\xi(x,t)}dx,
\end{align*}
where $\xi(x,t)=\varepsilon(x,t)+q(x)$. Making the substitution $y=(x-\hat
{x})/\hat{\sigma}$, it holds
\begin{equation}
\label{4.12}
\Psi_{2}(t,\alpha)=\hat{\sigma}^{\alpha+1}e^{K(\hat{x},t)}
\int_{-\frac{L^{\frac{1}{3}}(t)}{\sqrt{2}}}^{\frac{L^{\frac{1}{3}}(t)}
{\sqrt{2}}}y^{\alpha}\exp\left(  -\frac{y^{2}}{2}-\frac{\hat{\sigma}^{3}y^{3}%
}{6}h^{(2)}(\hat{x})+\xi(\hat{\sigma}y+\hat{x},t)\right)  dy,
\end{equation}
since $h^{\prime}(\hat{x})=1/\hat{\sigma}^{2}$.\newline

On $\left\lbrace y: y\in\left(-L^{\frac{1}{3}}(t)/\sqrt{2},L^{\frac{1}{3}}(t)/\sqrt{2}\right)\right\rbrace$, by Lemma \ref{LEM4.3}, we have
\[
\left\vert h^{(2)}(\hat{x})\hat{\sigma}^{3}y^{3}\right\vert \leq\left\vert
h^{(2)}(\hat{x})\hat{\sigma}^{3}L(t)\right\vert /2^{\frac{3}{2}}
\underset{t\rightarrow\infty}{\longrightarrow}0.
\]
Perform the first order Taylor expansion
\[
\exp\left(  -\frac{h^{(2)}(\hat{x})\hat{\sigma}^{3}}{6}y^{3}+\xi(\hat{\sigma
}y+\hat{x},t)\right)  \underset{t\rightarrow\infty}{=}1-\frac{h^{(2)}(\hat
{x})\hat{\sigma}^{3}}{6}y^{3}+\xi(\hat{\sigma}y+\hat{x},t)+o_{1}(t,y),
\]
where
\begin{equation}
o_{1}(t,y)=o\left(  -\frac{h^{(2)}(\hat{x})\hat{\sigma}^{3}}{6}y^{3}+\xi
(\hat{\sigma}y+\hat{x},t)\right)  . \label{4.13}%
\end{equation}

We obtain
\[
\int_{-\frac{L^{\frac{1}{3}}(t)}{\sqrt{2}}}^{\frac{L^{\frac{1}{3}}(t)}%
{\sqrt{2}}}y^{\alpha}\exp\left(  -\frac{y^{2}}{2}-\frac{\hat{\sigma}^{3}y^{3}%
}{6}h^{(2)}(\hat{x})+\xi(\hat{\sigma}y+\hat{x},t)\right)  dy =: T_{1}%
(t,\alpha)+T_{2}(t,\alpha),
\]
where $T_{1}(t,\alpha)$ is defined in (\ref{4.1}) and
\begin{equation}
T_{2}(t,\alpha):=\int_{-\frac{L^{\frac{1}{3}}(t)}{\sqrt{2}}}^{\frac
{L^{\frac{1}{3}}(t)}{\sqrt{2}}}\left(  \xi(\hat{\sigma}y+\hat{x}%
,t)+o_{1}(t,y)\right)  y^{\alpha}e^{-\frac{y^{2}}{2}}dy.
\end{equation}

Using (\ref{4.13}) we have for $t$ large enough
\begin{align*}
\left\vert T_{2}(t,\alpha)\right\vert  &  \leq\sup_{y\in\lbrack-\frac
{L^{\frac{1}{3}}(t)}{\sqrt{2}},\frac{L^{\frac{1}{3}}(t)}{\sqrt{2}}]}\left\vert
\xi(\hat{\sigma}y+\hat{x},t)\right\vert \int_{-\frac{L^{\frac{1}{3}}(t)}%
{\sqrt{2}}}^{\frac{L^{\frac{1}{3}}(t)}{\sqrt{2}}}\left\vert y\right\vert
^{\alpha}e^{-\frac{y^{2}}{2}}dy\\
&  \hspace{0.5cm}+\int_{-\frac{L^{\frac{1}{3}}(t)}{\sqrt{2}}}^{\frac
{L^{\frac{1}{3}}(t)}{\sqrt{2}}}\left(  \left\vert o\left(  \frac{h^{(2)}%
(\hat{x})\hat{\sigma}^{3}}{6}y^{3}\right)  \right\vert +\left\vert o(\xi
(\hat{\sigma}y+\hat{x},t))\right\vert \right)  |y|^{\alpha}e^{-\frac{y^{2}}%
{2}}dy,
\end{align*}
where $\sup_{y\in\lbrack-L^{\frac{1}{3}}(t)/\sqrt{2},L^{\frac{1}{3}}%
(t)/\sqrt{2}]}\left\vert \xi(\hat{\sigma}y+\hat{x},t)\right\vert \leq
\sup_{y\in\lbrack-L(t),L(t)]}\left\vert \xi(\hat{\sigma}y+\hat{x}%
,t)\right\vert $ since $L^{\frac{1}{3}}(t)/\sqrt{2}\leq L(t)$ holds for $t$
large enough. Thus
\begin{align*}
\left\vert T_{2}(t,\alpha)\right\vert  &  \leq2\sup_{y\in\lbrack
-L(t),L(t)]}\left\vert \xi(\hat{\sigma}y+\hat{x},t)\right\vert \int%
_{-\frac{L^{\frac{1}{3}}(t)}{\sqrt{2}}}^{\frac{L^{\frac{1}{3}}(t)}{\sqrt{2}}%
}\left\vert y\right\vert ^{\alpha}e^{-\frac{y^{2}}{2}}dy\\
&  \hspace{1cm}+\left\vert o\left(  \frac{h^{(2)}(\hat{x})\hat{\sigma}^{3}}%
{6}\right)  \right\vert \int_{-\frac{L^{\frac{1}{3}}(t)}{\sqrt{2}}}%
^{\frac{L^{\frac{1}{3}}(t)}{\sqrt{2}}}\left\vert y\right\vert ^{3+\alpha
}e^{-\frac{y^{2}}{2}}dy\\
&  \underset{t\rightarrow\infty}{=}\left\vert o\left(  \frac{h^{(2)}(\hat
{x})\hat{\sigma}^{3}}{6}\right)  \right\vert \left(  \int_{-\frac{L^{\frac
{1}{3}}(t)}{\sqrt{2}}}^{\frac{L^{\frac{1}{3}}(t)}{\sqrt{2}}}\left\vert
y\right\vert ^{\alpha}e^{-\frac{y^{2}}{2}}dy+\int_{-\frac{L^{\frac{1}{3}}%
(t)}{\sqrt{2}}}^{\frac{L^{\frac{1}{3}}(t)}{\sqrt{2}}}\left\vert y\right\vert
^{3+\alpha}e^{-\frac{y^{2}}{2}}dy\right)  ,
\end{align*}
where the last equality holds from Lemma \ref{LEM4.4}. Since the integrals in
the last equality are both bounded, it holds
\begin{equation}
T_{2}(t,\alpha)\underset{t\rightarrow\infty}{=}o(h^{(2)}(\hat{x})\hat{\sigma
}^{3}).
\end{equation}

When $\alpha$ is even, using (\ref{T1pair}) and Lemma \ref{LEM4.3}
\begin{equation}
\left\vert \frac{T_{2}(t,\alpha)}{T_{1}(t,\alpha)}\right\vert \leq\frac{\vert
h^{(2)}(\hat{x})\hat{\sigma}^{3}\vert}{\sqrt{2\pi}M_{\alpha}}
\underset{t\rightarrow\infty}{\longrightarrow}0.\label{T2a}
\end{equation}

When $\alpha$ is odd, using (\ref{T1impair}), we get
\begin{equation}
\frac{T_{2}(t,\alpha)}{T_{1}(t,\alpha)}\underset{t\rightarrow\infty}{=}%
-\frac{6}{\sqrt{2\pi}M_{\alpha+3}}o(1)\underset{t\rightarrow\infty
}{\longrightarrow}0.\label{T2b}
\end{equation}

Now with $\alpha\in\mathbb{N}$, by (\ref{T2a}) and (\ref{T2b})
\[
T_{2}(t,\alpha)\underset{t\rightarrow\infty}{=}o(T_{1}(t,\alpha)),
\]
which, combined with (\ref{4.12}), yields
\begin{equation}
\Psi_{2}(t,\alpha)=c\hat{\sigma}^{\alpha+1}e^{K(\hat{x},t)}T_{1}%
(t,\alpha)(1+o(1)). \label{4.14}%
\end{equation}

\vspace{0.4cm}

\textit{Step 3:} The Three Chords Lemma implies, for $x\mapsto K(x,t)$ concave and $(x,y,z)\in\mathbb{R}_{+}^{3}$ such that $x<y<z$
\begin{equation}
\frac{K(y,t)-K(z,t)}{y-z}\leq\frac{K(x,t)-K(z,t)}{x-z}\leq\frac{K(x,t)-K(y,t)}{x-y}. \label{4.15}
\end{equation}

Since $x\mapsto K(x,t)$ is concave and attains its maximum in $\hat{x}$, we
consider two cases: $x<\hat{x}$ and $x\geq\hat{x}$. After some calculus using (\ref{4.15}) in each case, we get
\begin{equation}
K(x,t)-K(\hat{x},t)\leq\frac{K(\hat{x}+sgn(x-\hat{x})\frac{L^{1/3}
(t)\hat{\sigma}}{\sqrt{2}})-K(\hat{x},t)}{sgn(x-\hat{x})\frac{L^{1/3}
(t)\hat{\sigma}}{\sqrt{2}}}(x-\hat{x}), \label{4.16}
\end{equation}
where
\begin{align*}
sgn(x-\hat{x})= \left\{
\begin{array}{ll}
1 & \mbox{if $x\geq \hat{x}$}\\
-1 & \mbox{if $x<\hat{x}$}
\end{array}
\right..
\end{align*}

Using Lemma \ref{LEM4.3}, a third-order Taylor expansion in the numerator of
(\ref{4.16}) gives
\[
K(\hat{x}+sgn(x-\hat{x})\frac{L^{1/3}(t)\hat{\sigma}}{\sqrt{2}})-K(\hat
{x},t)\leq-\frac{1}{4}h^{\prime}(\hat{x})L^{2/3}(t)\hat{\sigma}^{2}=-\frac
{1}{4}L^{2/3}(t),
\]
which yields
\[
K(x,t)-K(\hat{x},t)\leq-\frac{\sqrt{2}}{4}\frac{L^{1/3}(t)}{\hat{\sigma}}\vert
x-\hat{x}\vert.
\]

Using (\ref{4.2}), we obtain for large enough fixed $\tau$
\begin{align*}
\vert\Psi_{3}(t,\alpha)\vert &  \leq 2\int_{x\not \in \hat{x}+I_{t},x>\tau
}\vert x-\hat{x}\vert^{\alpha}e^{K(x,t)}dx\\
&  \leq 2e^{K(\hat{x},t)}\int_{\vert x-\hat{x}\vert>\frac{L^{1/3}
(t)\hat{\sigma}}{\sqrt{2}},x>\tau}\vert x-\hat{x}\vert^{\alpha}\exp\left(
-\frac{\sqrt{2}}{4}\frac{L^{1/3}(t)}{\hat{\sigma}}\vert x-\hat{x}\vert\right)
dx\\
&  =2e^{K(\hat{x},t)}\hat{\sigma}^{\alpha+1}\left[  \int_{\frac{L^{1/3}
(t)}{\sqrt{2}}}^{+\infty}y^{\alpha}e^{-\frac{\sqrt{2}}{4}L^{1/3}(t)y}
dy+\int_{\frac{\tau-\hat{x}}{\hat{\sigma}}}^{-\frac{L^{1/3}(t)}{\sqrt{2}}
}(-y)^{\alpha}e^{\frac{\sqrt{2}}{4}L^{1/3}(t)y}dy\right] \\
&  := 2e^{K(\hat{x},t)}\hat{\sigma}^{\alpha+1}(I_{\alpha}+J_{\alpha}).
\end{align*}

It is easy but a bit tedious to show by recursion that
\begin{align*}
I_{\alpha}  &  =\int_{\frac{L^{1/3}(t)}{\sqrt{2}}}^{+\infty}y^{\alpha}%
\exp\left(  -\frac{\sqrt{2}}{4}L^{1/3}(t)y\right)  dy\\
&  =\exp(-\frac{1}{4}L^{2/3}(t))\sum_{i=0}^{\alpha}2^{\frac{4i+3-\alpha}{2}%
}L^{\frac{\alpha-(2i+1)}{3}}(t)\frac{\alpha!}{(\alpha-i)!}\\
&  \underset{t\rightarrow\infty}{\sim}2^{\frac{3-\alpha}{2}}\exp(-\frac{1}%
{4}L^{2/3}(t))L^{\frac{\alpha-1}{3}}(t)
\end{align*}
and
\begin{align*}
J_{\alpha}  &  =\int_{\frac{\tau-\hat{x}}{\hat{\sigma}}}^{-\frac{L^{1/3}%
(t)}{\sqrt{2}}}(-y)^{\alpha}\exp\left(  \frac{\sqrt{2}}{4}L^{1/3}(t)y\right)
dy\\
&  =I_{\alpha}-\exp\left(  \frac{\sqrt{2}}{4}L^{1/3}(t)\frac{\tau-\hat{x}%
}{\hat{\sigma}}\right)  \sum_{i=0}^{\alpha}\left(  \frac{\hat{x}-\tau}%
{\hat{\sigma}}\right)  ^{\alpha-i}L^{-\frac{i+1}{3}}(t)2^{\frac{3i+3}{2}}%
\frac{\alpha!}{(\alpha-i)!}\\
&  =I_{\alpha}+M(t),
\end{align*}
with $\hat{x}/\hat{\sigma}\in RV((1+1/\beta)/2)$ when $h\in RV(\beta)$ and
$\hat{x}/\hat{\sigma}\in RV(1/2)$ when $h$ is rapidly varying. Moreover,
$\tau-\hat{x}<0$, thus $M(t)\underset{t\rightarrow\infty}{\longrightarrow}0$
and we have for some positive constant $Q_1$
\[
\vert\Psi_{3}(t,\alpha)\vert\leq Q_1e^{K(\hat{x},t)}\hat{\sigma}^{\alpha+1}
\exp(-\frac{1}{4}L^{2/3}(t))L^{\frac{\alpha-1}{3}}(t).
\]

With (\ref{4.14}), we obtain for some positive constant $Q_2$
\[
\left\vert\frac{\Psi_{3}(t,\alpha)}{\Psi_{2}(t,\alpha)}\right\vert\leq
\frac{Q_2\exp(-\frac{1}{4}L^{2/3}(t))L^{\frac{\alpha-1}{3}}(t)}{\vert
T_{1}(t,\alpha)\vert}.
\]

In Step 1, we saw that $T_{1}(t,\alpha)\underset{t\rightarrow\infty}{\sim
}\sqrt{2\pi}M_{\alpha}$, for $\alpha$ even and $T_{1}(t,\alpha)\underset{t\rightarrow\infty}{\sim}-\frac{h^{(2)}(\hat{x})\hat{\sigma}^{3}}{6}\sqrt{2\pi}M_{\alpha+3}$, for $\alpha$ odd. Hence for $\alpha$ even and $t$ large enough
\begin{equation}
\left\vert\frac{\Psi_{3}(t,\alpha)}{\Psi_{2}(t,\alpha)}\right\vert\leq
Q_3\frac{\exp(-\frac{1}{4}L^{2/3}(t))L^{\frac{\alpha-1}{3}}(t)}{\sqrt{2\pi
}M_{\alpha}}\underset{t\rightarrow\infty}{\longrightarrow}0, \label{4.56}
\end{equation}
and for $\alpha$ odd and $t$ large enough
\[
\left\vert\frac{\Psi_{3}(t,\alpha)}{\Psi_{2}(t,\alpha)}\right\vert\leq
Q_4\frac{\exp(-\frac{1}{4}L^{2/3}(t))L^{\frac{\alpha-1}{3}}(t)}{\frac{\vert
h^{(2)}(\hat{x})\hat{\sigma}^{3}\vert}{6}\sqrt{2\pi}M_{\alpha+3}},
\]
for positive constants $Q_3$ and $Q_4$.\\

As in Lemma \ref{LEM4.3}, we have
\[
\vert h^{(2)}(\hat{x})\hat{\sigma}^{3}\vert\in RV\left(\frac{\theta}{\beta}
+\frac{3}{2\beta}-\frac{3}{2}\right) \mbox{ if $h\in RV(\beta)$}
\]
and
\[
\vert h^{(2)}(\hat{x})\hat{\sigma}^{3}\vert\in RV\left(-\frac{1}{2}\right)
\mbox{ if $h$
is rapidly varying}.
\]
Let us denote
\[
\vert h^{(2)}(\hat{x})\hat{\sigma}^{3}\vert=t^{\rho}L_{1}(t),
\]
for some slowly varying function $L_{1}$ and $\rho<0$ defined as
\[
\rho= \left\{
\begin{array}
[c]{ll}%
\frac{\theta}{\beta}+\frac{3}{2\beta}-\frac{3}{2} & \mbox{ if $h\in
RV(\beta)$}\\
-\frac{1}{2} & \mbox{ if $h$ is rapidly varying}
\end{array}
\right.  .
\]

We have for some positive constant $C$
\[
\left\vert \frac{\Psi_{3}(t,\alpha)}{\Psi_{2}(t,\alpha)}\right\vert \leq
C\exp\left(  -\frac{1}{4}L^{2/3}(t)-\rho\log t-\log L_{1}(t)\right)
L^{\frac{\alpha-1}{3}}(t)\underset{t\rightarrow\infty}{\longrightarrow}0,
\]
since $-(\log t)^{2}/4-\rho\log t-\log L_{1}(t)\underset{t\rightarrow
\infty}{\sim}-(\log t)^{2}/4\underset{t\rightarrow\infty}{\longrightarrow
}-\infty$.

Hence we obtain
\begin{equation}
\Psi_{3}(t,\alpha)\underset{t\rightarrow\infty}{=}o(\Psi_{2}(t,\alpha)).
\label{4.18}
\end{equation}

The proof is completed by combining (\ref{4.3}), (\ref{4.11}), (\ref{4.14})
and (\ref{4.18}).
\end{proof}

\vspace{0.2cm}

\begin{proof}
[Proof of Theorem \ref{THM3.1}]By Lemma \ref{LEM4.5}, if $\alpha=0$, it holds
\[
T_{1}(t,0)\underset{t\rightarrow\infty}{\longrightarrow}\sqrt{2\pi},
\]
since $L(t)\underset{t\rightarrow\infty}{\longrightarrow}\infty$. Approximate the moment generating function of $X$ 
\begin{equation}
\Phi(t) =\Psi(t,0)\underset{t\rightarrow\infty}{=}\hat{\sigma}e^{K(\hat{x},t)}
T_{1}(t,0)(1+o(1))\underset{t\rightarrow\infty}{=}\sqrt{2\pi}\hat{\sigma} e^{K(\hat{x},t)}(1+o(1)). \label{4.19}
\end{equation}

If $\alpha=1$, it holds
\[
T_{1}(t,1)\underset{t\rightarrow\infty}{=}-\frac{h^{(2)}(\hat{x})\hat{\sigma
}^{3}}{6}M_{4}\sqrt{2\pi}(1+o(1)),
\]
where $M_{4}=3$ denotes the fourth order moment of the standard normal
distribution. Consequently, we obtain
\begin{equation}
\Psi(t,1)\underset{t\rightarrow\infty}{=}-\sqrt{2\pi}\hat{\sigma}
^{2}e^{K(\hat{x},t)}\frac{h^{(2)}(\hat{x})\hat{\sigma}^{3}}{2}
(1+o(1))\underset{t\rightarrow\infty}{=}-\Phi(t)\frac{h^{(2)}(\hat{x}
)\hat{\sigma}^{4}}{2}(1+o(1)), \label{4.20}%
\end{equation}
which, together with the definition of $\Psi(t,\alpha)$, yields
\[
\int_{0}^{\infty}xe^{tx}p(x)dx=\Psi(t,1)+\hat{x}\Phi(t)\underset{t\rightarrow
\infty}{=}\left(  \hat{x}-\frac{h^{(2)}(\hat{x})\hat{\sigma}^{4}}
{2}(1+o(1))\right)  \Phi(t).
\]
Hence we get
\begin{equation}
m(t)=\frac{\int_{0}^{\infty}xe^{tx}p(x)dx}{\Phi
(t)}=\hat{x}-\frac{h^{(2)}(\hat{x})\hat{\sigma}^{4}}{2}(1+o(1)). \label{4.21}
\end{equation}
By Lemma \ref{LEM4.3}, we obtain
\begin{equation}
m(t)\underset{t\rightarrow\infty}{\sim}\hat{x}=\psi(t). \label{4.22}
\end{equation}

If $\alpha=2$, it follows
\[
T_{1}(t,2)\underset{t\rightarrow\infty}{=}\sqrt{2\pi}(1+o(1)).
\]
Thus we have
\begin{equation}
\Psi(t,2)\underset{t\rightarrow\infty}{=}\hat{\sigma}^{2}\Phi(t)(1+o(1)).
\label{4.23}%
\end{equation}
Using (\ref{4.20}), (\ref{4.21}) and (\ref{4.23}), it follows
\begin{align*}
&  \int_{0}^{\infty}(x-m(t))^{2}e^{tx}p(x)dx=\int_{0}^{\infty}(x-\hat{x}
+\hat{x}-m(t))^{2}e^{tx}p(x)dx\\
&  =\Psi(t,2)+2(\hat{x}-m(t))\Psi(t,1)+(\hat{x}-m(t))^{2}\Phi(t)\\
&  \underset{t\rightarrow\infty}{=}\hat{\sigma}^{2}\Phi(t)(1+o(1))-\hat
{\sigma}^{2}\Phi(t)\frac{(h^{(2)}(\hat{x})\hat{\sigma}^{3})^{2}}
{4}(1+o(1))\underset{t\rightarrow\infty}{=}\hat{\sigma}^{2}\Phi(t)(1+o(1)),
\end{align*}
where the last equality holds since $\vert h^{(2)}(\hat{x}
)\hat{\sigma}^{3}\vert\underset{t\rightarrow\infty}{\longrightarrow}0$ by
Lemma \ref{LEM4.3}.

Hence we obtain
\begin{equation}
s^{2}(t)=\frac{\int_{0}^{\infty}(x-m(t))^{2}e^{tx}p(x)dx}{\Phi(t)}\underset{t\rightarrow\infty}{\sim}\hat{\sigma}^{2}=\psi^{\prime}(t). \label{4.24}
\end{equation}

If $\alpha=3$, it holds
\[
T_{1}(t,3)=-\frac{h^{(2)}(\hat{x})\hat{\sigma}^{3}}{6}\int_{-\frac{L^{\frac
{1}{3}}(t)}{\sqrt{2}}}^{\frac{L^{\frac{1}{3}}(t)}{\sqrt{2}}}y^{6}e^{-\frac{y^{2}}{2}}dy.
\]

Thus we have
\begin{align}
\Psi(t,3) &  =-\sqrt{2\pi}\hat{\sigma}^{4}e^{K(\hat{x},t)}\frac{h^{(2)}
(\hat{x})\hat{\sigma}^{3}}{6}\int_{-\frac{L^{\frac{1}{3}}(t)}{\sqrt{2}}
}^{\frac{L^{\frac{1}{3}}(t)}{\sqrt{2}}}\frac{1}{\sqrt{2\pi}}y^{6}
e^{-\frac{y^{2}}{2}}dy\label{4.25}\\
&  \underset{t\rightarrow\infty}{=}-M_{6}\frac{h^{(2)}(\hat{x})\hat{\sigma
}^{6}}{6}\Phi(t)(1+o(1)),\nonumber
\end{align}
where $M_{6}=15$ denotes the sixth order moment of standard normal
distribution. Using (\ref{4.20}), (\ref{4.21}), (\ref{4.23}) and (\ref{4.25}), we have
\begin{align*}
&\int_{0}^{\infty}(x-m(t))^{3}e^{tx}p(x)dx=\int_{0}^{\infty}(x-\hat{x}+\hat{x}-m(t))^{3}e^{tx}p(x)dx\\
&=\Psi(t,3)+3(\hat{x}-m(t))\Psi(t,2)+3(\hat{x}-m(t))^{2}\Psi(t,1)+(\hat
{x}-m(t))^{3}\Phi(t)\\
&\underset{t\rightarrow\infty}{=}-h^{(2)}(\hat{x})\hat{\sigma}^{6}\Phi(t)(1+o(1))-h^{(2)}(\hat{x})\hat{\sigma}^{6}\Phi(t)\frac{(h^{(2)}(\hat{x})\hat{\sigma}^{3})^{2}}{4}(1+o(1))\\
&\underset{t\rightarrow\infty}{=}-h^{(2)}(\hat{x})\hat{\sigma}^{6}\Phi(t)(1+o(1)),
\end{align*}
where the last equality holds since $|h^{(2)}(\hat{x})\hat{\sigma}^{3}|\underset{t\rightarrow\infty}{\longrightarrow}0$ by Lemma \ref{LEM4.3}.
Hence we get
\begin{equation}
\mu_{3}(t) =\frac{\int_{0}^{\infty
}(x-m(t))^{3}e^{tx}p(x)dx}{\Phi(t)}\underset{t\rightarrow\infty}{\sim}-h^{(2)}(\hat{x})\hat{\sigma}^{6}=\frac{\psi^{(2)}(t)}{(\psi^{\prime}(t))^{3}}(\psi^{\prime}(t))^{3}=\psi^{(2)}(t).\label{4.26}
\end{equation}

We now consider $\alpha=j>3$ for even $j$. Using (\ref{4.21}) and Lemma
\ref{LEM4.5}, we have
\begin{align}
&\int_{0}^{\infty}(x-m(t))^{j}e^{tx}p(x)dx=\int_{0}^{\infty}(x-\hat{x}+\hat{x}-m(t))^{j}e^{tx}p(x)dx\label{4.27}\\
&=\sum_{i=0}^{j}\binom{j}{i}\left(\frac{h^{(2)}(\hat{x})\hat{\sigma}^{4}}{2}\right)^{i}\hat{\sigma}^{j-i+1}e^{K(\hat{x},t)}T_{1}(t,j-i)(1+o(1)),\nonumber
\end{align}
with
\begin{align*}
T_{1}(t,j-i)&=\left\{
\begin{array}{ll}
\int_{-\frac{L^{\frac{1}{3}}(t)}{\sqrt{2}}}^{\frac{L^{\frac{1}{3}}(t)}
{\sqrt{2}}}y^{j-i}e^{-\frac{y^{2}}{2}}dy & \mbox{for even $i$}\\
-\frac{h^{(2)}(\hat{x})\hat{\sigma}^{3}}{6}\int_{-\frac{L^{\frac{1}{3}}
(t)}{\sqrt{2}}}^{\frac{L^{\frac{1}{3}}(t)}{\sqrt{2}}}y^{3+j-i}e^{-\frac{y^{2}
}{2}}dy & \mbox{for odd $i$}
\end{array}
\right.  \\
&\underset{t\rightarrow\infty}{=}\left\{
\begin{array}{ll}
\sqrt{2\pi}M_{j-i}(1+o(1)) & \mbox{if $i$ is even}\\
-\sqrt{2\pi}\frac{h^{(2)}(\hat{x})\hat{\sigma}^{3}}{6}M_{3+j-i} &\mbox{if $i$
is odd}
\end{array}
\right..
\end{align*}
Using (\ref{4.19}), we obtain
\begin{align*}
&\int_{0}^{\infty}(x-m(t))^{j}e^{tx}p(x)dx\\
&\underset{t\rightarrow\infty}{=}\sum_{i=0}^{j}\binom{j}{i}\left(
\frac{h^{(2)}(\hat{x})\hat{\sigma}^{4}}{2}\right)^{i}\Phi(t)\times\\
&\hspace{0.2cm}\left[\hat{\sigma}^{j-i}M_{j-i}(1+o(1))\mathbb{I}
_{even~i}-\frac{h^{(2)}(\hat{x})\hat{\sigma}^{4}}{2}\sigma^{j-i-1}
\frac{M_{3+j-i}}{3}(1+o(1))\mathbb{I}_{odd~i}\right]\\
&\underset{t\rightarrow\infty}{=}\sum_{k=0}^{j/2}\binom{j}{2k}\left(
\frac{h^{(2)}(\hat{x})\hat{\sigma}^{4}}{2}\right)^{2k}\Phi(t)\hat{\sigma
}^{j-2k}M_{j-2k}(1+o(1))\\
&\hspace{0.2cm}-\sum_{k=0}^{j/2-1}\binom{j}{2k+1}\left(\frac{h^{(2)}
(\hat{x})\hat{\sigma}^{4}}{2}\right)^{2(k+1)}\Phi(t)\hat{\sigma}
^{j-2k-2}\frac{M_{3+j-2k-1}}{3}(1+o(1))\\
&\underset{t\rightarrow\infty}{\sim}\hat{\sigma}^{j}\Phi(t)\times\\
&\left(M_{j}+\sum_{k=1}^{j/2}\binom{j}{2k}(h^{(2)}(\hat{x})\hat{\sigma
}^{3})^{2k}\frac{M_{j-2k}}{2^{2k}}-\sum_{k=0}^{j/2-1}\binom{j}{2k+1}
(h^{(2)}(\hat{x})\hat{\sigma}^{3})^{2(k+1)}\frac{M_{3+j-2k-1}}{3\times
2^{2(k+1)}}\right)\\
&\underset{t\rightarrow\infty}{=}M_{j}\hat{\sigma}^{j}\Phi(t)(1+o(1)),
\end{align*}
since $|h^{(2)}(\hat{x})\hat{\sigma}^{3}|\underset{t\rightarrow\infty
}{\longrightarrow}0$ by Lemma \ref{LEM4.3}. Hence we get for even $j$
\begin{equation}
\mu_{j}(t) =\frac{\int_{0}^{\infty}(x-m(t))^{j}e^{tx}p(x)dx}{\Phi(t)}\underset{t\rightarrow\infty}{\sim}M_{j}\hat{\sigma}^{j}\underset{t\rightarrow\infty}{\sim}M_{j}s^{j}(t),\label{4.28}
\end{equation}
by (\ref{4.24}).\newline

To conclude, we consider $\alpha=j>3$ for odd $j$. (\ref{4.27}) holds true with
\begin{align*}
T_{1}(t,j-i) &=\left\{
\begin{array}{ll}
\int_{-\frac{L^{\frac{1}{3}}(t)}{\sqrt{2}}}^{\frac{L^{\frac{1}{3}}(t)}
{\sqrt{2}}}y^{j-i}e^{-\frac{y^{2}}{2}}dy & \mbox{for odd $i$}\\
-\frac{h^{(2)}(\hat{x})\hat{\sigma}^{3}}{6}\int_{-\frac{L^{\frac{1}{3}}
(t)}{\sqrt{2}}}^{\frac{L^{\frac{1}{3}}(t)}{\sqrt{2}}}y^{3+j-i}e^{-\frac{y^{2}
}{2}}dy & \mbox{for even $i$}
\end{array}
\right. \\
&\underset{t\rightarrow\infty}{=}\left\{
\begin{array}{ll}
\sqrt{2\pi}M_{j-i}(1+o(1)) & \mbox{if $i$ is odd}\\
-\sqrt{2\pi}\frac{h^{(2)}(\hat{x})\hat{\sigma}^{3}}{6}M_{3+j-i} &\mbox{if $i$
is even}
\end{array}
\right..
\end{align*}
Thus, with the same tools as above, some calculus and making use of
(\ref{4.28}),
\[
\int_{0}^{\infty}(x-m(t))^{j}e^{tx}p(x)dx\underset{t\rightarrow\infty}{=}
\frac{M_{j+3}-3jM_{j-1}}{6}\times(-h^{(2)}(\hat{x})\hat{\sigma}^{j+3})\Phi(t).
\]
Hence we get for odd $j$
\begin{align}
\label{4.29}
\mu_{j}(t) &=\frac{\int_{0}^{\infty}(x-m(t))^{j}e^{tx}p(x)dx}{\Phi(t)}\underset{t\rightarrow\infty}{\sim}\frac{M_{j+3}-3jM_{j-1}}{6}\times(-h^{(2)}(\hat{x})\hat{\sigma}^{j+3})\\
&\underset{t\rightarrow\infty}{\sim}\frac{M_{j+3}-3jM_{j-1}}{6}\mu
_{3}(t)s^{j-3}(t),\nonumber\\
\end{align}
by (\ref{4.24}) and (\ref{4.26}).\newline

The proof is complete by considering (\ref{4.22}), (\ref{4.24}), (\ref{4.26}), (\ref{4.28}) and (\ref{4.29}).\newline
\end{proof}

\begin{proof}
[Proof of Theorem \ref{THM3.2}]It is proved incidentally in (\ref{4.19}).
\end{proof}

\begin{proof}[Proof of Theorem \ref{Thm3.4}] Consider the moment generating
function of the random variable
\[
Y_{t}:=\frac{\mathcal{X}_{t}-m(t)}{s(t)}.
\]
It holds for any $\lambda$
\begin{align*}
\log E\exp\lambda Y_{t}  & =-\lambda\frac{m(t)}{s(t)}+\log\frac{\Phi\left(  t+\frac{\lambda}{s(t)}\right)}{\Phi(t)}\\
& =\frac{\lambda^2}{2}\frac{s^{2}\left(t+\theta\frac{\lambda}{s(t)}\right)}
{s^{2}(t)}=\frac{\lambda^2}{2}\frac{\psi^{\prime}\left(t+\theta\frac{\lambda\left(1+o(1)\right)}{\sqrt{\psi^{\prime}(t)}}\right)}{\psi^{\prime}(t)}\left(1+o(1)\right)
\end{align*}
as $t\rightarrow\infty,$ for some $\theta\in\left(0,1\right)$ depending on
$t$, where we used Theorem \ref{THM3.1}. Now making use of Corollaries
\ref{COR2.2} and \ref{COR2.3} it follows that
\[
\lim_{t\rightarrow\infty}\log E\exp\lambda Y_{t}=\frac{\lambda^2}{2},
\]
which proves the claim.
\end{proof}

Aknowledgement: The authors thank an anonymous referee for comments and suggestions which helped greatly to the improvement on a previous version of this paper.

\bibliographystyle{plain}
\bibliography{ColloqueDHV2013-biblio}

\end{document}